\documentclass[12pt,psamsfonts]{amsart}
\usepackage{amsthm}
\usepackage{amssymb}
\usepackage{amsmath}
\usepackage{graphicx}
\usepackage{amscd}
\usepackage{amsfonts}
\usepackage{amsbsy}
\usepackage[T1]{fontenc}
\usepackage[english]{babel}
\usepackage{xcolor}
\usepackage{a4wide}
\usepackage{calligra}
\usepackage{mathtools}

\usepackage{nicematrix}
\usepackage{arydshln}
\usepackage{vaucanson-g}
 \usepackage{tikz}
\usetikzlibrary{matrix,decorations.pathreplacing}

\usepackage{epsfig}
\usepackage{amssymb}

\newtheorem{theorem}{Theorem}

\newtheorem{proposition}[theorem]{Proposition}

\newcommand{\N}{\mathbb N}

\newcommand{\uu}{{\bf u}}
\newcommand{\vv}{{\bf v}}

\newcommand{\A}{{\mathcal A}}

\newcommand{\F}{\mathcal F}

\newcommand{\G}{\mathcal G}

\author[B. San Mart\'in]{Bernardo San Mart\'in}
\address{ 
Departamento de Matem\'aticas, Universidad Cat\'olica del Norte, Antofagasta, Chile}
\email{sanmarti@ucn.cl}

\author[V.F. Sirvent]{V\'{\i}ctor F. Sirvent}
\address{ 
Departamento de Matem\'aticas, Universidad Cat\'olica del Norte, Antofagasta, Chile}
\email{victor.sirvent@ucn.cl}

\title{On the Lorenz-Fibonacci sequences and substitutions}

\date{\today}



\begin{document}

\begin{abstract}
In the present article, we consider two families of integer sequences, 
the $(k,r)$-Lorenz-Fibonacci sequences of the first and second kind, whose characteristic polynomial is
$$
p_{k,r}(x)=x^k-x^{k-1}-\cdots- x^r+x^{r-1}+\cdots +x+1,
$$
where $k\geq 2r+1$, and $k\geq 3$.
These families include the well-known $k$-bonacci sequence, when $r=0$.
These sequences arise naturally in the study of the dynamics of the Lorenz attractor. 
We introduce  two families of substitutions (in an alphabet of $k$ symbols) so that they are associated with each of the families of integer sequences and share the same polynomial.  We study the main combinatorial properties of these substitutions.
\end{abstract}

\maketitle

\section{Introduction}

The geometrical model  of  the Lorenz attractor, introduced by Guckenheimer~(\cite{G,afraimovich,GW}), plays  a fundamental role in the theory  of dynamical systems, see for instance~\cite{viana,ghys} and  within references.
In~\cite{SS},
the authors  studied the symbolic dynamics of this model satisfying certain combinatorial and dynamical  assumptions and showed that a special class of shift spaces  arises.
These shifts are  associated with the family of polynomials
\begin{equation}\label{eqn:poly}
p_{k,r}(x):=x^k-x^{k-1}-\cdots- x^r+x^{r-1}+\cdots +x+1,
\end{equation}
where $k,r$ are non-negative integers such that $k\geq 2r+1$ and $k\geq 3$.

In this article, we study two different families of integer sequences having~(\ref{eqn:poly}) as characteristic polynomial. We show some   combinatorial properties of these integer sequences. 
An important fact about these families is that they include the well-known $k$-bonacci sequence. Some members of these families of sequences are not listed in~\cite{OEIS}, at the time of writing this article.

We also present a family of morphisms or substitutions on an alphabet of $k$ symbols that  reproduces the
integer sequences mentioned and these substitutions
are associated with the dynamics of the Lorenz shifts. We explore some combinatorial and dynamical properties of these morphisms in Theorems~\ref{thm:1} and~\ref{thm:2}.

The topics discussed in the article involve integer sequences, combinatorics on words, substitution dynamical systems and potential connections with the dynamics of the Lorenz attractor.
This article is organized as follows: in Section~\ref{s:sequence} we present the Lorenz-Fibonacci sequences of first and second kind, and show some basic properties and compute their generating function.
In Section~\ref{s:zeta} we introduce the preliminary concepts about substitutions and present the substitution $\zeta_{k,r}$ which is associated with the $(k,r)$-Lorenz-Fibonacci sequence of the first kind and show in Theorem~\ref{thm:1} some of its basic combinatorial properties.
In Section~\ref{s:phi} we define the substitution $\phi_{k,r}$   which is associated with the $(k,r)$-Lorenz-Fibonacci sequence of the second kind and state in Theorem~\ref{thm:2} some of its basic combinatorial properties. 
These two families  of substitutions are structurally different (their corresponding graphs are not isomorphic), in spite of having the same incidence matrix.
We conclude with Sections~\ref{s:remarks} and~\ref{s:conclusions}, where we list some remarks,  open problems, and summarize the main results of the article.

\section{The $(k,r)$ Lorenz-Fibonacci sequences}\label{s:sequence}

The  \emph{$(k,r)$-Lorenz-Fibonacci recurrence relation} is the
linear  recurrence relation given in~(\ref{eqn:lr}), with  $k\geq 2r+1$, $k\geq 3$, and $r\geq 0$; which  has $p_{k,r}(x)$ as its characteristic polynomial.
\begin{multline}\label{eqn:lr}
L_{k,r}(k+n)=L_{k,r}(n+k-1)+\cdots +L_{k,r}(n+r)-\\
-L_{k,r}(n+r-1)-\cdots-L_{k,r}(n+1)
-L_{k,r}(n).
\end{multline}
Note that for $r=0$,~(\ref{eqn:lr}) reduces to the $k$-bonacci recurrence relation.

We shall consider a different set of integer initial values, depending on the applications considered in sections~\ref{s:zeta} and~\ref{s:phi}.

One of the simplest initial values are:
$$
L_{k,r}(0)=L_{k,r}(1)=\cdots=L_{k,r}(k-2)=0, \quad L_{k,r}(k-1)=1.
$$
The corresponding sequence $\{L_{k,r}(n)\}_{n\geq 0}$ is called the \emph{$(k,r)$-Lorenz-Fibonacci sequence of first kind}, or simply 
\emph{$(k,r)$-Lorenz-Fibonacci sequence}.
Note that $\{L_{k,0}\}_{n\geq 0}$ is the well-known $k$-bonacci sequence, introduced in~\cite{miles} and widely studied in different contexts, see for instance~\cite{flores,Koshy,RS,si:1997,BFMT,PW,KP,KL,GL,DD}.
The first values of some of these families of sequences are presented in Table~\ref{tab:1}. Note that the majority of them are not listed in~\cite{OEIS}.

\begin{proposition}\label{prop:Ln}
Let $k\geq 2r+1$, and $k\geq 4$, $r\geq 1$. Then
$$
L_{k,r}(n)=\left\{
\begin{array}{ll}
2^{n-k}, &  k\leq n< 2k-r;\\
2^{n-k}-(n-2k+r+2) 2^{n-2k+r}, &  2k-r\leq n \leq 2k-1.
\end{array}\right.
$$
 \end{proposition}
\begin{proof}
Observe that for $k<n\leq 2k-r$, the  terms with a minus sign in  the recurrence relation~(\ref{eqn:lr}) are null. Thus, for  $0\leq j\leq k-r-1$, 
$$
L_{k,r}(k+j)=L_{k,r}(k-1)+L_{k,r}(k)+\cdots+L_{k,r}(k+j-1)=1+1+2+\cdots+2^{j-1}=2^{j}.
$$

For $n=2k-r+j$, with $0\leq j\leq r-1$.
\begin{multline*}
L_{k,r}(2k-r+j)=\underbrace{L_{k,r}(2k-r+j-1)+\cdots + L_{k,r}(k+j)}_{A(j)}+\\-\underbrace{(L_{k,r}(k+j-1) +\cdots +L_{k,r}(k+r-j))}_{B(j)},
\end{multline*}
and not all the elements of $B(j)$ are null.
For $j=0$:
$$
L_{k,r}(2k-r)=L_{k,r}(2k-r-1)+\cdots+L_{k,r}(k)-L_{k,r}(k-1)=2^{k-r-1}+\cdots+2+1-1=2^{k-r}-2.
$$
For $j=1$, we get $L_{k,r}(2k-r+1)=A(1)-B(1)$ with:
$$
A(1)=L_{k,r}(k+1)+L_{k,r}(k+2)+\cdots+L_{k,r}(2k-r)=(2+2^2+\cdots+2^{2k-r-1})+(2^{2k-r}-2)=2(2^{k-r}-2),
$$
and
$
B(1)=L_{k,r}(k)+L_{k,r}(k-1)=1+1=2.
$
It yields $$L_{k,r}(2k-r+1)=A(1)-B(1)=2(2^{k-r+1}-2)-(2)=2^{k-r+1}-6.$$
Proceeding recursively, for $2\leq j\leq r-1$, it yields
$L_{k,r}(2k-r+j)=2^{2k-r+j}-(2+j)2^{j}$.
\end{proof}

\begin{proposition}\label{prop:gf}
Let  $F_{k,r}(x):=\sum_{n\geq 0}L_{k,r}(n)x^n$ be the generating function of the sequence $\{L_{k,r}(n)\}_{n\geq 0}$. 
It is given by 
$$
F_{k,r}(x)=\dfrac{x^{k-1}(1-x)}{1-2x+2x^{k-r+1}-x^{k+1}}
$$
\end{proposition}
\begin{proof}
In this proof, we simplify the notation, we drop the subindices $(k,r)$, which are fixed throughout the proof.

Using the initial conditions, it follows
$
F(x)=\sum_{m\geq 0}L(n)x^n=x^{k-1}+\sum_{n\geq k} L(n)x^n.
$

The recurrence relation~(\ref{eqn:lr}) yields
\begin{equation}\label{eqn:gf-r2}
\sum_{n\geq 0} L(n+k)x^{n+k}=\sum_{n\geq 0} \sum_{j=r}^{k-1} L(n+j)x^{n+k} - \sum_{n\geq 0}\sum_{j=0}^{r-1} L(n+j)x^{n+k}.
\end{equation}
Considering the left hand side of~(\ref{eqn:gf-r2})
$$
\sum_{n\geq 0} L(n+k)x^{n+k}=F(x)-\sum_{j=0}^{k-1}L(j)x^j=F(x)-x^{k-1}.
$$

Analyzing the first sum of the right hand side of~(\ref{eqn:gf-r2}):
\begin{eqnarray*}
\sum_{n\geq 0} \sum_{j=r}^{k-1} L(n+j)x^{n+k}&=&\sum_{j=r}^{k-1}x^{k-j}\sum_{m\geq j }L(m)x^m\\
&=&\sum_{j=r}^{k-1}x^{k-j}(F(x)-\sum_{m=0}^{j-1} L(m)x^m)\\
&=&F(x)\sum_{j=r}^{k-1}x^{k-j} \,\, \left(\text{Since } L(j)=0, \text{ for }\,0\leq j< k-1\right).
\end{eqnarray*}
By a similar computation, the second term of the right hand side of the equality~(\ref{eqn:gf-r2}) is equal to $F(x)\sum_{j=0}^{r-1}x^{k-j}$.
Hence
$$
F(x)-x^{k-1}= F(x)\left(\sum_{j=r}^{k-1}x^{k-j}-\sum_{j=0}^{r-1}x^{k-j}\right).
$$
Thus
$F(x)=x^{k-1}/(1-q(x))$,
where
$$
q(x)=(x+\cdots+x^{k-r})-(x^{k-r+1}+\cdots+x^k)=\dfrac{x(1-x^{k-r})}{1-x}+\dfrac{x^{k-r+1}(1-x^r)}{1-x}=\dfrac{x-2x^{k-r+1}+x^{k+1}}{1-x}.
$$
So 
$$
F(x)=\dfrac{x^{k-1}}{1-q(x)}=\dfrac{x^{k-1}(1-x)}{1-2x+2x^{k-r+1}-x^{k+1}}.
$$
\end{proof}

We would like to mention that 
the computation of the generating function of a sequence given a linear recurrence relation is a standard technique, however we presented it here for the sake of completeness, and also, the function $F_{k,r}(x)$ has a ``neat" expression.

We remark that the generating function  $F_{k,r}(x)$ can be used in presenting an alternative proof of Proposition~\ref{prop:Ln}.

We also consider a different set  of initial values for the recurrence relation~(\ref{eqn:lr}):
\begin{equation}\label{eqn:ell-ic}
\ell_{k,r}(n)=\left\{
\begin{array}{ll}
2^{n}, &  0\leq n\leq k-r;\\
 2^{n}-(n-k+r+1) 2^{n-k+r}, &  k-r+1\leq n \leq k-1.
\end{array}\right.
\end{equation}
The corresponding sequence
 $\{\ell_{k,r}(n)\}_{n\geq 0}$ is called \emph{$(k,r)$-Lorenz-Fibonacci sequence of second kind}.
 In the following sections we show the combinatorial significance and properties of  these sequences.   The first terms of some of these sequences are listed in Table~\ref{tab:2}.
For $r=0$, clearly $\ell_{k,0}(n)=L_{k,0}(n+k)$, i.e., it is the shifted $k$-bonacci sequence. 

\begin{proposition}\label{prop:gf2}
Let  $\F_{k,r}(x):=\sum_{n\geq 0}\ell_{k,r}(n)x^n$ be the generating function of the sequence $\{\ell_{k,r}(n)\}_{n\geq 0}$. 
It is given by 
$$
\F_{k,r}(x)=\dfrac{(1-x)H(x)}{1-2x+2x^{k-r+1}-x^{k+1}},
$$
where $H(x)$ is the polynomial of degree $k-1$:
\begin{multline*}
H(x)=\sum_{n=1}^{k-1}2^nx^n-\sum_{n=k-r+1}^{k-1} (n-k+r+1)2^{n-k+r}x^n
+\sum_{j=0}^{r-1}x^{k-j}\sum_{n=0}^{j-1} 2^n x^n+\\
-\sum_{j=r}^{k-1}x^{k-j}\left(\sum_{n=0}^{k-r} 2^n x^n +\sum_{n=k-r+1}^{j-1}\left(2^n-(n-k+r+1)2^{n-k+r}\right)x^n\right).
\end{multline*}
\end{proposition}
The proof of this proposition is similar to  Proposition~\ref{prop:gf}, with the difference of using the new set of initial values~(\ref{eqn:ell-ic}); for this reason, we do not present it.

\begin{table}[h]
    \centering
    \begin{tabular}{|l|c|c|}
\hline
$(k,r)$ & $L_{k,r}(n)$ & OEIS\\ \hline
$(3,0)$ & $0,0,1,1,2,4,7,13,24,44,81,149,\ldots $&A00073 \,  (Tribonacci sequence)\\ \hline
$(4,0)$ & $0,0,0,1,1,2,4,8,15,29,56,108,\ldots$  & A00078 \, (Tetranacci sequence)\\ \hline
$(5,1)$& $0,0,0,0,1,1,2,4,8,14,27,51,96,180\ldots$ & A06516  (shifted)\\ \hline
$(6,1)$& $0,0,0,0,0,1,1,2,4,8,16,32,61,119,232,\ldots$ & Not listed\\ \hline
$(7,1)$&  $0,0,0,0,0,0,1,1,2,4,8,16,32,62,123,243,480\ldots$  &  Not listed  \\ \hline
$(5,2)$& $0,0,0,0,1,1,2,4,6,8,15,23,36,60,\ldots $ &  Not listed\\ \hline
$(6,2)$ & $0,0,0,0,0,1,1,2,4,8,14,26,49,91\ldots $ &  Not listed\\ \hline
$(7,2)$&  $0,0,0,0,0,0,1,1,2,4,8,16,30,58,113,219,424\cdots$    &  Not listed \\ \hline
$(7,3)$ & $0,0,0,0,0,0,1,1,2,4,8,14,26,48,89,163,300\ldots$ &  Not listed\\ \hline
$(8,2)$ & $0,0,0,0,0,0,0,1,1,2,4,8,16,32,62,122,241,475\cdots$ &  Not listed \\ \hline
$(8,3)$ & $0,0,0,0,0,0,0,1,1,2,4,8,16,30,58,112,217,419,\cdots$ &  Not listed\\  \hline
$(9,3)$ & $0,0,0,0,0,0,0,0,1,1,2,4,8,16,32,62,122,240,473,\ldots$&  Not listed \\  \hline
\end{tabular}
\caption{Some examples of $L_{k,r}(n)$:  $(k,r)$-Lorenz-Fibonacci sequences}
    \label{tab:1}
\end{table}

\begin{table}[h]
    \centering
\begin{tabular}{|l|c|c|}
\hline
$(k,r)$ & $\ell_{k,r}(n)$ & OEIS\\ \hline
$(4,1)$ & $1,2,4,7,12,21,36,62,107,184,317,546\ldots$ & Not listed \\ \hline
$(5,1)$& $1,2,4,8,15,28,53,100,188,354,667,1256,\ldots$ & A118870 \\ \hline
$(6,1)$ & $1,2,4,8,16,31,60,117,228,444,864,1682\ldots$ & Not listed \\ \hline
$(6,2)$ & $1,2,4,8,15,28,52,97,180,334,620,\ldots$ & Not listed\\ \hline
$(7,1)$ & $1,2,4,8,16,32,63,124,245,484,956,\ldots$ & Not listed \\ \hline
$(7,2)$ &1,2,4,8,16,31,60,116,225,436,844,1634,\ldots & Not listed \\ \hline
$(8,1)$ & 1,2,4,8,16,32,64,127,252,501,996,1980,\ldots& Not listed \\ \hline
$(8,2)$ & 1,2,4,8,16,32,63,124,244,481,948,1868,\ldots & Not listed \\ \hline
$(8,3)$ & 1,2,4,8,16,31,60,116,224,433,836,1614,\ldots& Not listed \\ \hline
\end{tabular}
\caption{Some examples of $\ell_{k,r}(n)$:  $(k,r)$-Lorenz-Fibonacci sequences of second kind.}
    \label{tab:2}
\end{table}

In the present article, we do not consider or compute the Binet formulae for these families of sequences.
In~\cite{SS}, we showed that the polynomial $p_{k,r}(x)$ has a single dominant root, which is a positive number, say $\rho_{k,r}$. 
Therefore $L_{k,r}(n)$ and $\ell_{k,r}(n)$ grow asymptotically  like  $\rho_{k,r}^n$.

\section{The Lorenz substitutions or morphisms}\label{s:zeta}

In the following lines we introduce some  preliminary notions about substitutions required throughout the article. For a general theory on the subject, see~\cite{BR,queffelec}.

Let $\A$ be a finite set, called the \emph{alphabet} and its elements are called \emph{symbols} or \emph{letters}.
 Denote by $\A^*=\cup_{n\geq 0}\A^n$ the set of all finite words over $\A$, where  $\A^0$ is  the  set formed by the empty word $\varepsilon$. 
Given a word  $U\in\A^*$, we write $|U|$ for its length. 

\smallskip
A  \emph{substitution} or \emph{morphism} on $\A$ is a map $\zeta: \A \to \A^*$. This map extends naturally  to $\A^*$ by  concatenation: for all  $U,V\in\A^*$, we define recursively $\zeta(UV)=\zeta(U)\zeta(V)$.
This map is extended to $\A^{\N}$ the set of one-sided infinite sequences on $\A$ in the following way: If $\vv=v_1v_2\cdots$, we define $\zeta(\vv)=\zeta(v_1)\zeta(v_2)\cdots$. Note we keep the same notation $\zeta$ for all these different maps obtained by the substitution in the different spaces defined in terms of $\A$. We are interested in finding the fixed or periodic points of the substitution, considered as maps on $\A^\N$.

\begin{proposition}[\cite{queffelec}, p. 126]\label{prop:fp}
Let $\zeta$ be a  substitution on the alphabet $\A$, such that
\linebreak
$\lim_{n\to\infty}|\zeta^n(a)|=\infty$, for all $a\in\A$.
Then there exists $\vv\in\A^\N$ and $m\geq 1$ such that $\vv=\zeta^m(\vv)$.
\end{proposition}

Note that  if $\vv=\zeta^m(\vv)$ and  $\vv=v_1v_2\cdots$ then  $\zeta^{nm}(v_1)$ is a  prefix of $\vv$  for all  $n\geq 1$. For this reason we   denote
$\vv=\lim_{n\to\infty}\zeta^{mn}(v_1)$.

\smallskip

An important object associated with a morphism $\zeta$ is its matrix $M_\zeta$, defined by $(M_\zeta)_{i,j} = |\zeta(j)|_i$; that is, the number of times the symbol $i$ appears in the word $\zeta(j)$, assuming that the alphabet is ordered as $\A = \{1, \dots, k\}$.
We say that a substitution $\zeta$ is \emph{primitive} if the matrix $M_\zeta$ is primitive, i.e, there exists $m>0$ such that all the entries of $M_\zeta^m$ are positive. The substitutions considered in this article are primitive.

We can associate a directed graph  to the substitution $\zeta$.
The vertices of the graph are the elements of the alphabet. 
Let $a,b\in\A$, if $b$ is present in $\zeta(a)$ then there is an oriented edge from $a$ to $b$. 
This graph is called the graph of the substitution $\zeta$, and it is denoted by $\G_\zeta$.
Clearly, by construction the adjacency matrix of $\G_\zeta$ coincides with the matrix $M_\zeta$.
Note that given a directed graph does not represent uniquely the substitution, since it does not  determine the order of the symbols in $\zeta(a)$, for any $a\in\A$, when its length is larger than $2$.
In order to have unicity, it is necessary  to introduce labels on the edges, However, we do not need this in the present article, for details see~\cite{BR}.
If $\zeta$ is primitive the $\G_\zeta$ is strongly connected.

Given a substitution $\zeta$ satisfying the hypothesis of Proposition~\ref{prop:fp}
 we can consider the integer sequence $\{g_n\}_{n\geq 0}$, given by $g_n=|\zeta^n(a)|$, for $a\in\A$,  $n\geq  1$, and $g_0=|\zeta^0(a)|=|a|=1$; obviously  this sequence depends (in general) on the choice of $a$.

\medskip

In this section,  we consider substitutions associated with the integer sequences introduced in the previous section.
We denote the  substitutions, considered in this section,  as $\zeta_{k,r}$, however we shall omit the subindex $(k,r)$ and use only $\zeta$, whenever the context is clear in order to avoid  a cumbersome notation in the identities.
We show in Theorem~\ref{thm:1} (item  (e)) that the sequence $\{|\zeta_{k,r}^n(1)|\}_{n\geq 0}$ is the shifted  sequence $\{L_{k,r}(m)\}_{m\geq 0}$.

 In order to improve the readability,
we separate the analysis  in three cases: $r=0$, $r=1$ and $r>1$.

\noindent
{\bf The $k$-bonacci case, $(k,0)$:} It corresponds to the well-known $k$-bonacci substitution given by:
$$
1\to 12,\quad 2\to 13, \quad\ldots\quad (k-1)\to 1k,\quad k\to 1,
$$
which has been widely studied, in different contexts, see for instance:~\cite{si:1997,emme,FR,ACS} . 
The characteristic polynomial of $\zeta_{k,0}$ is $p_{k,0}(x)$, since  its companion matrix is $M_{\zeta_{k,0}}$,
for details see the proof of Theorem~\ref{thm:2}-(c).
Its graph is depicted in Figure~\ref{fig:graph-k-bonacci}.
Note that the morphism $\zeta$  has a unique fixed point: $\uu=\lim_{m\to\infty}\zeta^{m}(1)$, since the image of each symbol starts with $1$.
Its main recurrence property is described in the following proposition.

\begin{proposition} For $n\geq 0$,
$$
\zeta^{n+k}(1)=\zeta^{n+k-1}(1)\cdots\zeta^{n+1}(1)\zeta^n(1).
$$
\end{proposition}
This is a well-known fact, however we present the proof here for the sake of completeness.
\begin{proof}
Consider
\begin{eqnarray*}
\zeta^k(1)&=&\zeta^{k-1}(12)=\zeta^{k-1}(1)\zeta^{k-1}(2)\\
&=&\zeta^{k-1}(1)\zeta^{k-2}(1)\zeta^{k-2}(3)\\
&\vdots & \\
&=&\zeta^{k-1}(1)\zeta^{k-2}(1)\zeta^{k-2}(1)\cdots\zeta(1)\zeta(k)\\
&=& \zeta^{k-1}(1)\zeta^{k-2}(1)\zeta^{k-2}(1)\cdots\zeta(1)1.
\end{eqnarray*}
After applying $n$-times the morphism $\zeta$,  the desired identity is obtained. 
\end{proof}

\begin{figure}
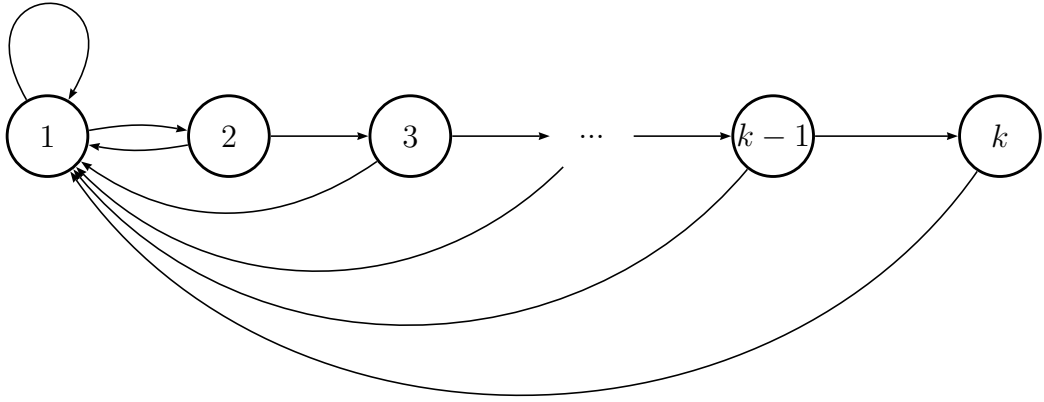

\begin{center}
\VCDraw{
\begin{VCPicture}{(0,-4)(24,5)}
\FixStateDiameter{1.85cm} \ChgStateLabelScale{1}
\State[1]{(2,2)}{A} \State[2]{(6,2)}{B} \State[3]{(10,2)}{C}
 \ChgStateLineStyle{none}
\State[...]{(14,2)}{D} \RstStateLineStyle \State[k-1]{(18,2)}{E} \State[k]{(23,2)}{F}
\ChgEdgeLabelScale{1} \LoopN[.5]{A}{}

\EdgeL{B}{C}{}\EdgeL{C}{D}{}\EdgeL{D}{E}{}
\EdgeL{E}{F}{}

\VArcR{arcangle=10}{A}{B}{}
\VArcL[0.45]{arcangle=10}{B}{A}{}

\VArcL[0.5]{arcangle=35}{C}{A}{}
\VArcL[0.45]{arcangle=45}{D}{A}{}
\VArcL[0.5]{arcangle=50}{E}{A}{}
\VArcL[0.45]{arcangle=55}{F}{A}{}

\end{VCPicture}}
\end{center}
\caption{The  $k$-bonacci graph: $\G_{\zeta_{k,0}}$.
\label{fig:graph-k-bonacci}}
\end{figure}

\medskip

\noindent
{\bf Case $(k,1)$:}
Let $k$ be an integer with $k\geq 4$. We define the substitution $\zeta_{k,1}$, on the alphabet $\{1,\ldots, k\}$ given by the rules:
\begin{equation*}
1\to 32,\quad 2\to 42,\quad 3\to 14,\quad 4\to 15,\quad \cdots,(k-2)\to 1(k-1),\quad (k-1)\to 1k,\quad k\to 1.
\end{equation*}

The characteristic polynomial of the matrix $M_{\zeta_{k,1}}$ is $p_{k,1}(x)$, see~\cite[Proposition 6]{SS}.
The graph associated  with the substitution $\zeta_{k,1}$ is depicted in Figure~\ref{fig:k-1}.
Observe that the structure of the graph (and the substitution) does not make sense (it is not primitive) if $k=3$.

\begin{figure}
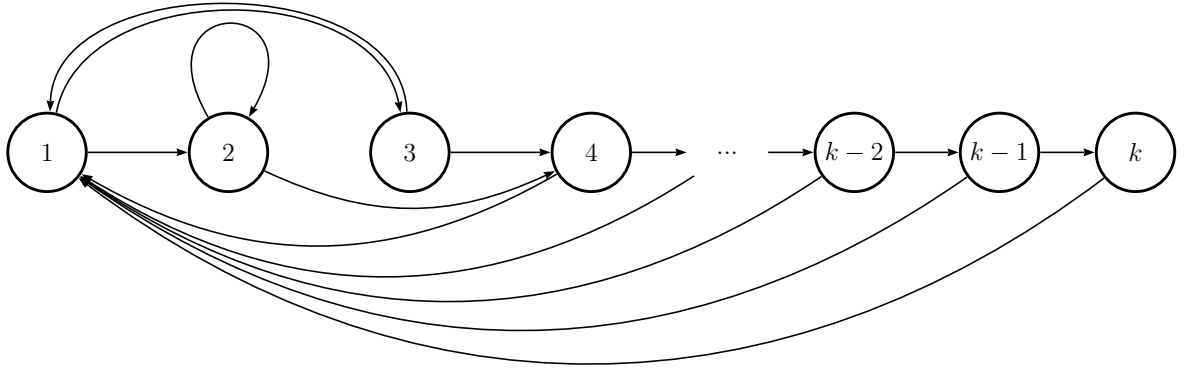

\begin{center}
\VCDraw{
\begin{VCPicture}{(0,-4)(27,5)}
\FixStateDiameter{1.8cm} \ChgStateLabelScale{0.8}
\State[1]{(2,1)}{A} \State[2]{(6,1)}{B} \State[3]{(10,1)}{C}
\State[4]{(14,1)}{D}
 \ChgStateLineStyle{none}
\State[...]{(17,1)}{E} \RstStateLineStyle \State[k-2]{(19.8,1)}{F} \State[k-1]{(23,1)}{G}\State[k]{(26,1)}{H}

 \LoopN[.5]{B}{}

\EdgeL{A}{B}{}
\EdgeL{C}{D}{}\EdgeL{D}{E}{}
\EdgeL{E}{F}{}\EdgeL{F}{G}{}\EdgeL{G}{H}{}

\VArcR{arcangle=-25}{B}{D}{}

\VArcR{arcangle=79}{A}{C}{}
\VArcL{arcangle=-88}{C}{A}{}\LabelL[0.5]{}


\VArcL[0.45]{arcangle=30}{D}{A}{}
\VArcL[0.5]{arcangle=33}{E}{A}{}
\VArcL[0.5]{arcangle=34}{F}{A}{}
\VArcL[0.5]{arcangle=35}{G}{A}{}
\VArcL[0.5]{arcangle=37}{H}{A}{}

\end{VCPicture}}
\end{center}
\caption{The $(k,1)$ Lorenz-Fibonacci graph: $\G_{\phi_{k,1}}$.
\label{fig:k-1}}
\end{figure}

\begin{proposition}
Let $\zeta=\zeta_{k,1}$ be the substitution defined above, with $k\geq 4$.
\begin{enumerate}
\item[(a)]
The substitution $\zeta$ has only two periodic points: 
$\uu=\lim_{m\to\infty}\zeta^{2m}(1)$ and $\vv=\lim_{m\to\infty}\zeta^{2m}(3)$, with $\zeta(\uu)=\vv$.

\item[(b)]
The recurrence of the iterates of the substitution are given by the relation:
$$
\left\{\begin{array}{l}
\zeta^n(1)=\zeta^{n-2}(1)\zeta^{n-3}(1)\cdots\zeta^{n-k}(1)\zeta^{n-1}(2),\\
\\
 \zeta^n(2)=\zeta^{n-1}(1)\cdots\zeta^{n-k+1}(1)\zeta^{n-1}(2),
 \end{array}\right.
$$
with $n>k\geq5$. And it is not possible to express $\zeta^n(1)$ as juxtaposition of words of the form $\zeta^j(1)$, with $0\leq j\leq n-1$.
\end{enumerate}
\end{proposition}

\begin{proof}

\textbf{(a):} 
Note that $\zeta^2(1)=\zeta(32)=1442$, so $\zeta^2$ (acting on $\A^N$) has a fixed point which starts with the symbol $1$. 
Similarly $\zeta^2(3)=\zeta(1)\zeta(4)=3215$, so $\zeta^2$ has a fixed point that starts with the symbol $3$.

On the other hand, the word $\zeta(j)=1(j+1)$, for $3\leq j\leq k-1$, and $\zeta(k)=1$ so $\zeta^m(j)$ does not have a fixed point, $4\leq j\leq k-1$, for $m\geq 1$.

For $j=2$, we have $\zeta(2)=42$, and $\zeta^i(4)$ starts with $2$ for $i\geq1$, Hence $\zeta^m$ does not have a fixed point that starts with $2$.
This completes the proof of statement (a).

\bigskip

\noindent
\textbf{(b):} Let $k\geq 5$,
consider $\zeta^k(1)=\zeta^{k-1}(32)=\zeta^{k-1}(3)\zeta^{k-1}(2)=\zeta^{k-2}(1)\zeta^{k-2}(4)\zeta^{k-1}(2)$, and
\begin{equation}\label{eqn:z4}
\zeta^{k-2}(4)=\zeta^{k-3}(1)\zeta^{k-3}(5)=
\cdots=\zeta^{k-3}(1)\zeta^{k-4}(1)\zeta^{k-5}(1)\cdots\zeta(1)1.
\end{equation}
Hence
\begin{equation}\label{eqn:recurrence-1}
\zeta^{k}(1)=\zeta^{k-2}(1)\zeta^{k-3}(1)\zeta^{k-4}(1)\zeta^{k-5}(1)\cdots\zeta(1)1\zeta^{k-1}(2).
\end{equation}
Using~(\ref{eqn:z4}) in the iterates of $2$, it yields
\begin{equation}\label{eqn:recurrence-2}
\zeta^{k-1}(2)=\zeta^{k-2}(4)\zeta^{k-2}(2)=\zeta^{k-3}(1)\zeta^{k-4}(1)\zeta^{k-5}(1)\cdots\zeta(1)1\zeta^{k-2}(2).
\end{equation}
Note that  the word $\zeta^{k-2}(2)$ can not be express   entirely as juxtaposition of $\zeta^j(1)$, with $0\leq j < k-2$, since  
\begin{eqnarray*}
\zeta^{k-2}(2)&=&\zeta^{k-3}(4)\zeta^{k-3}(2)=\zeta^{k-4}(1)\zeta^{k-5}(1)\zeta^{k-5}(1)\cdots \zeta^2(1)1k\,\zeta^{k-3}(2),\\
\zeta^{k-3}(2)&=&\zeta^{k-5}(1)\zeta^{k-5}(1)\cdots \zeta^2(1)1(k-1)\zeta^{k-4}(2),\\
&\vdots& \\
\zeta^2(2) &=& 1542.
\end{eqnarray*}

After  iterating $n$ and  $n+1$ times the identities~(\ref{eqn:recurrence-1}) and~(\ref{eqn:recurrence-2}), respectively; it yields
\begin{align*}
\zeta^{k+n}(1)&=\zeta^{n+k-2}(1)\zeta^{n+k-3}(1)\cdots\zeta^{n+1}(1)\zeta^n(1)\zeta^{n+k-1}(2),\\
\zeta^{k+n}(2)&=\zeta^{n+k-2}(1)\cdots \zeta^{n+1}(1)\zeta^{n+k-1}(2).
\end{align*}
This concludes the proof of  statement (b).
\end{proof}

The \emph{return time of a vertex $v$ in a directed graph} $\G$ is the length of a path on $\G$ starting from $v$ and finishing in $v$. It is straight-forward to check that  the minimal return time of the vertex $1$ in $\G_{\zeta_{k,1}}$ is $3$. 
On the other hand, \emph{the minimal return time of a symbol $a\in\A$ in a substitution $\zeta$} is the smaller $j \geq 1$, such that $\zeta^j(a)$ starts with the symbol $a$. 
The minimal return time of the symbol $1$ for the substitution $\zeta_{k,1}$ is equal to $3$, it coincides with the minimal return time of $1$ in $\G_{\zeta_{k,1}}$. 
If we consider other substitutions that have the same graph, the return time of $1$ increases, we leave the details of this fact to the reader. 
For this reason we chose the substitution $\zeta_{k,1}$ and not the other ones having the same graph.
In the  case  $(k,0)$ the return time of the vertex  $1$ in the graph $\G_{k,0}$ is equal $1$, which is the same return time of $a=1$, for the  substitution $\zeta_{k,0}$. Note that the other substitutions having the  same  graph  as $\G_{k,0}$ the minimal return time of $a=1$ in the substitution is larger.
 
We remark that the  symbol $1$ is important in the applications to the dynamics of the Lorenz attractor, studied in~\cite{SS}, since it is associated with the generator  of the Markov partition of the map.

\medskip

\noindent
{\bf The general case $(k,r)$:} 
Let $k$ be an integer with $k\geq 2r+2$. We define the substitution $\zeta_{k,r}$, on the alphabet $\{1,\ldots, k\}$ as
\begin{multline*}
1\to 32,\quad 2\to 42,\quad 3\to 52,\quad 4\to 62,\quad \cdots,
(2r)\to (2r+2)2,\\
\quad (2r+1)\to 1(2r+2),\quad (2r+2)\to 1(2r+3),\quad \cdots,
\quad (k-1)\to 1k,\quad k\to 1.
\end{multline*}
 Note that the case $(k,r)$ includes all the  previous cases as particular situations.
 Observe that if $k\leq 2r+1$, the graph and the substitution has a very different structure, for this reason we consider only the situation $k\geq 2r+2$.

\begin{figure}
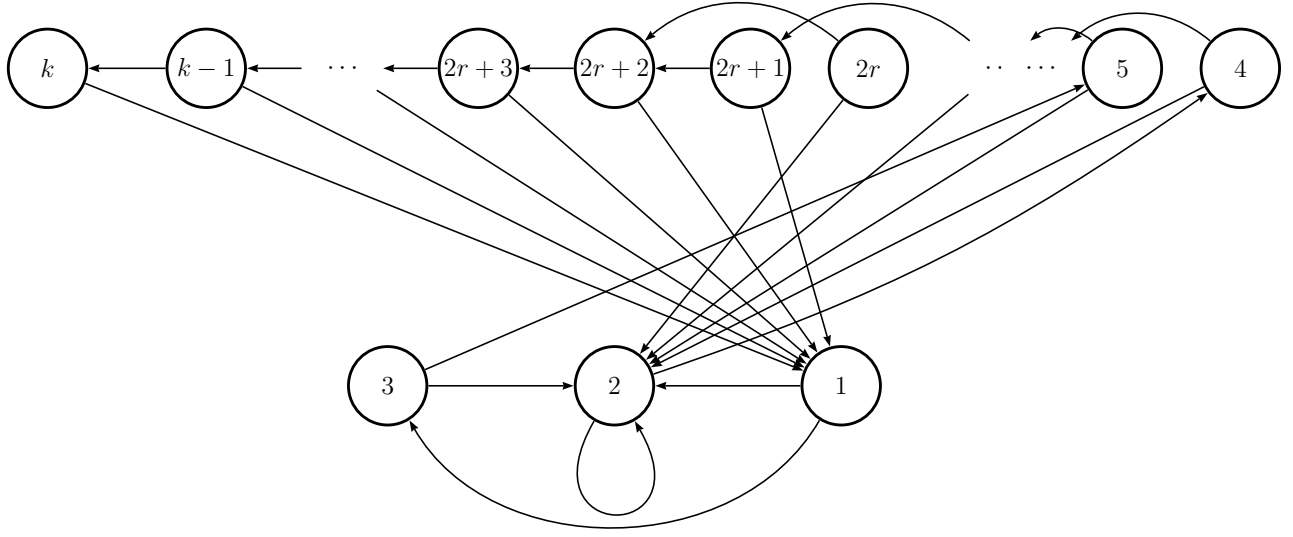

\begin{center}
\VCDraw{
\begin{VCPicture}{(0.5,-3)(29,10)}
\FixStateDiameter{1.8cm} \ChgStateLabelScale{0.8}
\State[k]{(1.5,8)}{A} \State[k-1]{(5,8)}{B}  \ChgStateLineStyle{none}\State[\cdots]{(8,8)}{C}\RstStateLineStyle 
\State[2r+3]{(11,8)}{D}

\State[2r+2]{(14,8)}{E} \State[2r+1]{(17,8)}{F} \State[2r]{(19.6,8)}{G} \ChgStateLineStyle{none}\State[\cdots]{(22.5,8)}{H}\State[\cdots]{(23.4,8)}{I}\RstStateLineStyle 
\State[5]{(25.2,8)}{J}\State[4]{(27.8,8)}{K}

\State[3]{(9,1)}{L}\State[2]{(14,1)}{M}\State[1]{(19,1)}{N}

 \LoopS[.5]{M}{}

\EdgeL{B}{A}{}
\EdgeL{C}{B}{}\EdgeL{D}{C}{}
\EdgeL{E}{D}{}\EdgeL{F}{E}{}
\EdgeL{N}{M}{}\EdgeL{L}{M}{}

\VArcR{arcangle=-40}{G}{E}{}
\VArcR{arcangle=-40}{H}{F}{}
\VArcR{arcangle=-40}{J}{H}{}
\VArcR{arcangle=-40}{K}{I}{}

\VArcR{arcangle=60}{N}{L}{}

\EdgeL{A}{N}{}\EdgeL{B}{N}{}\EdgeL{C}{N}{}\EdgeL{D}{N}{}\EdgeL{E}{N}{}\EdgeL{F}{N}{}

\EdgeL{L}{J}{}

\EdgeL{K}{M}{}\EdgeL{J}{M}{}\EdgeL{H}{M}{}\EdgeL{G}{M}{}
\VArcR{arcangle=-8}{M}{K}{}



\end{VCPicture}}
\end{center}
\caption{The $(k,r)$ Lorenz-Fibonacci graph: $\G_{\zeta_{k,r}}$.
\label{fig:k-r}}
\end{figure}

We consider the example: $k=7$, $r=2$; the substitution $\zeta=\zeta_{7,2}$ is given by
$$
1\to 32,\quad 2\to 42,\quad 3\to 52,\quad 4\to 62, \quad 5\to 16,\quad 6\to 17, \quad 7\to 1.
$$ 
Observe that
\begin{eqnarray*}
1&\to &32\to 5242\to 16426242,\\
3 &\to& 52\to 1642\to 32176242,\\
5 &\to& 16\to 3217\to 5242321.
\end{eqnarray*}
Therefore $\uu_j=\lim_{m\to\infty}\zeta^{3m}(2j-1)$, for $1\leq j\leq 3$ are the periodic points.

Note
$$
\zeta^5(1)=\zeta^2(1)\zeta(1)1\zeta^3(2)\zeta^4(2), \quad \zeta^4(2)=\zeta(1)1\zeta^2(2)\zeta^3(2),
$$
and
$$
\zeta^6(1)=\zeta^3(1)\zeta^2(1)\zeta(1)\zeta^4(2)\zeta^5(2), \quad \zeta^5(2)=\zeta^2(1)\zeta(1)\zeta^3(2)\zeta(1)1\zeta^2(2)\zeta^3(2).
$$
The initial values of the sequence $\{|\zeta^n(1)|\}_n$ are
$|\zeta^n(1)|=2^n$, for $0\leq n\leq 4$,
\begin{eqnarray*}
|\zeta^5(1)|&=&|\zeta^2(1)|+|\zeta(1)|+1+|\zeta^3(2)|+|\zeta(1)|+1+|\zeta^2(2)|+|\zeta^3(2)|\\
&=&2^2+2+1+2^3+2+1+2^2+2^3=30.
\end{eqnarray*}
Similarly $|\zeta^5(1)|=58$.
The first values of the sequence $\{|\zeta^n(1)|\}_n$ are
$$
1,\quad 2,\quad 4,\quad 8, \quad 16,\quad 30,\quad 58,\quad 113,\quad \ldots
$$
On the other hand,   the recurrence of the words of the iterates are given by the identities, where $n\geq 0$:
$$
\begin{array}{c}
\zeta^{n+7}(1)=\zeta^{n+6}(3)\zeta^{n+6}(2)=\zeta^{n+5}(5)\zeta^{n+5}(2)\zeta^{n+6}(2)=\zeta^{n+4}(1)\zeta^{n+3}(1)\zeta^{n+2}(1)\zeta^{n+5}(2)\zeta^{n+6}(2),\\
\\
\zeta^{n+7}(2)=\zeta^{n+6}(4)\zeta^{n+6}(2)=\zeta^{n+5}(6)\zeta^{n+5}(2)\zeta^{n+6}(2)=\zeta^{n+4}(1)\zeta^{n+3}(1)\zeta^{n+5}(2)\zeta^{n+6}(2).
\end{array}
$$

\medskip

\begin{theorem}\label{thm:1}
Let $\zeta=\zeta_{k,r}$ be the substitution defined above, where $k>2r+1$ and $r\geq 1$.
\begin{enumerate}
\item[(a)]
The substitution $\zeta$ does not have fixed points and it has  $r+1$ periodic points: $\uu_1,\uu_2,\ldots,\uu_{r+1}$, such that
$\uu_j=\lim_{m\to\infty}\zeta^{m(r+1)}(2j-1)$, for $1\leq j\leq r+1$, and $ \zeta(\uu_i)=\uu_{i+1}$, for $1\leq i\leq r$ and $\zeta(\uu_{r+1})=\uu_1$.

\item[(b)]
The recurrence of the iterates of the substitution is given by the relation:
$$
\left\{\begin{array}{l}
\zeta^n(1)=\underbrace{\zeta^{n-r-1}(1)\zeta^{n-r-2}(1)\cdots\zeta^{n+r-k}(1)}_{k-2r}\underbrace{\zeta^{n-r}(2)\zeta^{n-r+1}(2)\cdots\zeta^{n-1}(2)}_{r},\\
\\
 \zeta^n(2)=\underbrace{\zeta^{n-r-1}(1)\zeta^{n-r-2}(1)\cdots\zeta^{n+r-k+1}(1)}_{k-2r-1}\underbrace{\zeta^{n-r}(2)\zeta^{n-r+2}(2)\cdots\zeta^{n-1}(2)}_{r}\\
 \end{array}\right.
$$
with $n\geq k$. And it is not possible to express $\zeta^n(1)$ only as juxtaposition of words of the form $\zeta^j(1)$, with $0\leq j\leq n-1$.
\item[(c)] 
The  return time of the symbol $1$ under the substitution $\zeta$ is $r+1$, which is equal to the minimal return time of the vertex $1$ in the graph $\G_{\zeta_{k,r}}$. 
Moreover,  the minimum return time for any other substitution sharing the graph $\G_{\zeta_{k,r}}$ is larger than $r+1$.
\item[(d)] The characteristic polynomial of the substitution matrix $M_{\zeta_{k,r}}$ is $p_{k,r}(x)$. 
\item[(e)] $L_{k,r}(n+k)=|\zeta^n_{k,r}(1)|$.
\end{enumerate}
\end{theorem}
\begin{proof}

\textbf{(a):} 
The substitution does not have a fixed point since the image of each symbol is a different symbol.

Consider the symbol $1$, the first letter of $\zeta^j(1)$ is $2j+1$, for $1\leq j\leq r$ and the image $\zeta{(2r+1)}=1(2r+2)$, therefore the first letter of the word  $\zeta^{r+1}(1)$ is $1$. It yields $\lim_{m\to\infty}\zeta^{(r+1)m}(1)$ is a periodic point of period $r+1$. This point is denoted by $\uu_1$.
Similar analysis is done for the symbols $2j-1$, with $2\leq j\leq r$+1; so, the points
$\uu_j=\lim_{m\to\infty}\zeta^{(r+1)m}(j)$ are periodic points of period $r+1$, and all of them are in the same orbit.

Note that if we take a letter of the form $(2j)$, the first letter of the  image under $\zeta$ is the  $(2j+2)$ (if $1\leq j< r-1$) and eventually  is $1$, and afterwards it never gets an even number as first  symbol. Hence the first letter of all the iterates of $\zeta^m(2j)$ is never $(2j)$, so it is not periodic.

Therefore the only periodic points of $\zeta$ are $\uu_j$, for $1\leq j\leq r+1$.
This proves statement (a).

\bigskip

\noindent
\textbf{(b):} 
Consider $\zeta^n(1)=\zeta^{n-1}(3)\zeta^{n-1}(2)$, and
$$
\zeta^{n-1}(3)=\zeta^{n-2}(5)\zeta^{n-2}(2)=\cdots=\zeta^{n-r}(2r+1)\zeta^{n-r}(2)\zeta^{n-r+1}(2)\cdots\zeta^{n-2}(2).
$$
Computing the recurrence of $\zeta^{n-r}(2r+1)$:
\begin{eqnarray*}
\zeta^{n-r}(2r+1)&=&\zeta^{n-r-1}(1)\zeta^{n-r-1}(2r+2)\\
 &=&\zeta^{n-r-1}(1)\zeta^{n-r-2}(1)\zeta^{n-r-2}(2r+3)\\
 &\vdots& \\
 &=&\zeta^{n-r-1}(1)\zeta^{n-r-2}(1)\cdots\zeta^{n+r-k}(1)
\end{eqnarray*}
Combining the two formulae, it yields
\begin{equation}\label{eqn:zn-1}
\zeta^n(1)=\zeta^{n-r-1}(1)\zeta^{n-r-2}(1)\cdots\zeta^{n+r-k}(1))\zeta^{n-r}(2)\zeta^{n-r+1}(2)\cdots\zeta^{n-1}(2).
\end{equation}

On the other hand: $\zeta^n(2)=\zeta^{n-1}(4)\zeta^{n-1}(2)$, and
$$
\zeta^{n-1}(4)=\zeta^{n-2}(6)\zeta^{n-2}(2)=\cdots=\zeta^{n-r}(2r+2)\zeta^{n-r}(2)\zeta^{n-r+1}(2)\cdots\zeta^{n-1}(2).
$$
Studying:
$$
\zeta^{n-r}(2r+2)=\zeta^{n-r-1}(1)\zeta^{n-r-1}(2r+3)=\cdots=
\zeta^{n-r-1}(1)\zeta^{n-r-2}(1)\cdots\zeta^{n+r-k+1}(1).
$$
Hence
\begin{equation}\label{eqn:zn-2}
 \zeta^n(2)=\zeta^{n-r-1}(1)\zeta^{n-r-2}(1)\cdots\zeta^{n+r-k+1}(1)\zeta^{n-r}(2)\zeta^{n-r+2}(2)\cdots\zeta^{n-1}(2).
 \end{equation}
 The expressions~(\ref{eqn:zn-1}) and~(\ref{eqn:zn-2}) describe the recurrence of $\zeta^n(1)$ and $\zeta^n(2)$.
 
 Observe that the recurrence of $\zeta^m(2)$ always contains expressions of the form $\zeta^i(2)$, thus the factors $\zeta^{n-r}(2),\ldots, \zeta^{n-1}(2)$ in~(\ref{eqn:zn-1}) cannot be reduced to iterates of the letter 1.
 This completes the proof of statement (b).
 
 \bigskip

\noindent
\textbf{(c):} 
The analysis done in (a) shows that the minimal return time of  the letter $1$ under the substitution is $r+1$, since $1$ is the first letter of  $\zeta(j)$, for $2r+1\leq j\leq k-1$. Therefore if the substitution is of the form:
$j \to (j+1)1$, for any $2r+1\leq j\leq k-1$ then the return time of $1$ increases.
Similarly, if $\zeta(i)= 2(2i+1)$ for any $i\in\{1,3,\ldots,2r-1\}$ or $\zeta(i)= 2(2i+2)$ for any $i\in\{2,4,\ldots,2r\}$, the return time of the  symbol $1$ increases, even  it could be infinite.

On the other hand, counting on the graph $\G_{\zeta_{k,r}}$ we observe  that  among the shortest
cycles that start and end in the vertex $1$ are:
$$
1\to3\to 5\to\cdots\to (2r+1)\to 1, \quad 1\to 2\to 4\to\cdots\to (2r+2)\to 1.
$$ 
The length of the first and second cycles are $r+1$ and $r+2$, respectively. Therefore the minimal return time of the vertex $1$ in the graph $\G_{\zeta_{k,r}}$ is $r+1$.
This completes the proof of statement (c).
\bigskip

\noindent
\textbf{(d):} 
An easy computation shows that $P M_{\zeta_{k,r}}P=N(k,r)$, where 
$P=(p_{i,j})$ is the $k\times k$ matrix whose entries are given by 
$p_{i,k-i+1}=1$ for $1\leq i\leq k$, and $p_{i,j}=0$ otherwise. And
 $N(k,r)=\left(
\begin{array}{cc}
N^1 & N^2 \\
N^3 & N^4%
\end{array}%
\right) $ , where

$$N^1=\left[
\begin{array}{ccccc}
0 & 1 & \cdots  & 0 & 0 \\
0 & 0 & \ddots  & 0 & 0 \\
\vdots  & \vdots  & \ddots  & \ddots  & \vdots  \\
0 & 0 & \cdots  & 0 & 1 \\
0 & 0 & \cdots  & 0 & 0%
\end{array}%
\right] _{k-2r-1},\quad
 N^2=\left[
\begin{array}{ccccc}
0 & 0 & 0 & \cdots  & 0 \\
0 & 0 & 0 & \cdots  & 0 \\
\vdots  & \vdots  & \vdots  & \ddots  & \vdots  \\
0 & 0 & 0 & \cdots  & 0 \\
1 & 1 & 0 & \cdots  & 0%
\end{array}%
\right] _{k-2r-1\times 2r+1},$$

$$\ N^3=\left[
\begin{array}{cccc}
0 & 0 & \cdots  & 0 \\
\vdots  & \vdots  & \ddots  & \vdots  \\
0 & 0 & \cdots  & 0 \\
1 & 1 & \ldots  & 1%
\end{array}%
\right] _{2r+1\times k-2r-1}
\quad
\text{ and }\quad 
N^4=\left[
\begin{array}{ccccccc}
0 & 0 & 1 & \cdots  & 0 & 0 & 0 \\
0 & 0 & 0 & \ddots  & 0 & 0 & 0 \\
0 & 0 & 0 & \ddots  & \ddots  & 0 & 0 \\
\vdots  & \vdots  & \vdots  & \ddots  & \ddots  & \ddots  & \vdots  \\
0 & 0 & 0 & \cdots  & 0 & 0 & 1 \\
0 & 1 & 1 & \cdots  & 1 & 1 & 1 \\
1 & 0 & 0 & \cdots  & 0 & 0 & 0%
\end{array}%
\right] _{2r+1}.$$

The characteristic polynomial of the matrix $N(k,r)$ is $p_{k,r}(x)$~(\emph{cf.}~\cite[Theorem 9]{SS}).
Thus,
$$
p_{k,r}(x)=\det(N(k,r)-xI_k)=\det(PM_{\zeta_{k,r}}P-xI_k)=\det(P(M_{\zeta_{k,r}}-xI_k)P)=\det(M_{\zeta_{k,r}}-xI_k).
$$
This completes the proof of statement (d).

\bigskip

\noindent
\textbf{(e):} 
  Since $p_{k,r}(x)$ is the characteristic polynomial of the matrix $M_{\zeta_{k,r}}$, hence the sequence $\{|\zeta_{k,r}^n(1)|\}_{n\geq 0}$ satisfies the linear recurrence relation~(\ref{eqn:lr}). 
  Checking its initial values:
  
  Let denote $b_n=|\zeta_{k,r}^n(1)|$ and $c_n$ the number of occurrence of the symbol $k$ in $\zeta_{k,r}^n(1)$.
  Since the length of the image of any letter different from $k$ is $2$ the length of the image of $k$ is equal $1$, it follows
  $ b_{n+1}=2b_{n}-c_n$.
 Since the first time k occurs in $\zeta^n(1)$ is at $n=k-r-1$, thus $c_n=0$ for 
 $0\leq n<k-r-1$.
 Note that $c_{k-r-1}=2$, because the symbol $(k-1)$ appears twice in $\zeta^{k-r-2}(1)$,  one occurrence is from the images of even symbols and the other from the odd symbols: $\zeta^{k-r-2}(1)=\zeta^{k-r-3}(3)\zeta^{k-r-3}(2)$, following the orbits of the symbols $3$ and $2$:
 $$
 3\to 5\to\cdots\to (2r+1)\to (2r+2)\to\cdots\to (k-1), \quad 2\to 4\to\cdots\to (2r+2)\to\cdots\to (k-1). 
 $$
 After the step $k-r$ the number of $k$'s doubles until it reaches $n=k-1$.
 Hence
 $$
 c_{n}=\left\{\begin{array}{ll}
 0, & 0\leq n < k-r-1;\\
 2, & n=k-r-1;\\
 2^{n-k+r+1}, & k-r\leq n\leq k-2;\\
 2^r-1, & n=k-1.
\end{array}\right.
 $$
 Since $b_0=1$, we get
 $$
 b_{n}=\left\{\begin{array}{ll}
 2^n, & 0\leq n \leq k-r-1;\\
 2^n-(n-k+r+2)2^{n-k+r}, & k-r\leq n\leq k-1.
\end{array}\right.
 $$
 
 Comparing with $L_{k,r}(n)$ in Proposition~\ref{prop:Ln}, it follows 
 $b_n=L_{k,r}(n+k)$, for $0\leq n\leq k-1$.
 This finishes the proof of  statement (e).
\end{proof}

Note that for $2r\leq k\leq 2r+1$ the substitution changes its structure, and some of the properties stated in Theorem~\ref{thm:2} are not valid in these cases.
\section{The direct Lorenz-Fibonacci substitutions}\label{s:phi}

In this section we introduce another  family of substitutions that share the incidence matrix of the substitutions $\zeta_{k,r}$ studied in the previous section, however they have a different combinatorial behaviour. 
We call these morphisms  the \emph{$(k,r)$-direct Lorenz-Fibonacci substitutions}, we denote them by $\phi_{k,r}$,  whenever the context is clear, we only use $\phi$. One important fact is that they are associated with the $(k,r)$-Lorenz-Fibonacci sequences of second kind (Theorem~\ref{thm:2}(d)). 

Let  $k\geq 2r+1$ and $r\geq 0$. We define the substitution $\phi_{k,r}$ on the alphabet $\{1,\ldots,k\}$ as
\begin{multline}\label{eqn:phi}
1\to12,\quad 2\to 13,\quad\cdots,\quad (k-r-1)\to1(k-r),\\
(k-r)\to (k-r+1),\quad (k-r+1)\to (k-r+2),\quad\cdots, \quad k-1\to k,\quad k\to (r+1).
\end{multline}

Note that for $(k,0)$, $\phi_{k,0}=\zeta_{k,0}$, i.e., it is  the classical $k$-bonacci substitution.
The graph associated with the substitution $\phi_{k,1}$ is depicted in Figure~\ref{fig:k-1-d} and the general case, for any $r$ is in Figure~\ref{fig:k-r-d}.

 \begin{figure}
\begin{center}
\VCDraw{
\begin{VCPicture}{(0,-3)(26,5)}
\FixStateDiameter{1.8cm} \ChgStateLabelScale{0.8}
\State[1]{(2,1)}{A} \State[2]{(6,1)}{B} \State[3]{(10,1)}{C}
 \ChgStateLineStyle{none}
\State[...]{(13,1)}{D} \RstStateLineStyle \State[k-2]{(17,1)}{E} \State[k-1]{(21,1)}{F}\State[k]{(25,1)}{G}

 \LoopN[.5]{A}{}

\EdgeL{B}{C}{}\EdgeL{C}{D}{}\EdgeL{D}{E}{}
\EdgeL{E}{F}{}\EdgeL{F}{G}{}

\VArcR{arcangle=10}{A}{B}{}\LabelL[0.5]{}
\VArcL[0.45]{arcangle=10}{B}{A}{}

\VArcL[0.5]{arcangle=35}{C}{A}{}
\VArcL[0.45]{arcangle=45}{D}{A}{}
\VArcL[0.5]{arcangle=50}{E}{A}{}
\VArcL[0.45]{arcangle=-25}{G}{B}{}

\end{VCPicture}}
\end{center}
\caption{The $(k,1)$ direct Lorenz-Fibonacci  graph: $\G_{\phi_{k,1}}$
\label{fig:k-1-d}}
\end{figure}

\begin{figure}
\begin{center}
\VCDraw{
\begin{VCPicture}{(0,0)(26,14)}
\FixStateDiameter{2.2cm} \ChgStateLabelScale{0.65}
\State[1]{(1,8)}{A} \State[2]{(5,8)}{B} \State[3]{(9,8)}{C} 
 \ChgStateLineStyle{none}\State[\cdots]{(12.5,8)}{D}\RstStateLineStyle 
\State[r+1]{(16,8)}{E} 
\ChgStateLineStyle{none}\State[\cdots]{(20,8)}{F}\RstStateLineStyle 
\State[k-r-1]{(24,8)}{G} 

\State[k-r]{(5,4)}{H} \State[k-r+1]{(9,2)}{I}
 \ChgStateLineStyle{none}\State[\cdots]{(13,2)}{J}\RstStateLineStyle 
\State[k-1]{(16,2)}{K}\State[k]{(21,2)}{L}

\LoopS[.5]{A}{}

\EdgeL{B}{C}{}\EdgeL{C}{D}{}\EdgeL{D}{E}{}\EdgeL{E}{F}{}\EdgeL{F}{G}{}
\EdgeL{H}{I}{}\EdgeL{I}{J}{}\EdgeL{J}{K}{}\EdgeL{K}{L}{}
\EdgeL{L}{E}{}\EdgeL{G}{H}{}


\VArcR{arcangle=10}{A}{B}{}
\VArcL[0.45]{arcangle=10}{B}{A}{}

\VArcR{arcangle=-30}{C}{A}{}
\VArcR{arcangle=-33}{D}{A}{}
\VArcR{arcangle=-34}{E}{A}{}
\VArcR{arcangle=-35}{F}{A}{}
\VArcR{arcangle=-36}{G}{A}{}

\end{VCPicture}}
\end{center}
\caption{The $(k,r)$ direct  Lorenz-Fibonacci graph: $\G_{\phi_{k,r}}$.
\label{fig:k-r-d}}
\end{figure}

We consider the example, let $k=7$ and $r=2$, so the substitution $\phi_{7,2}$ is given by:
$$
1\to 12,\quad 2\to 13, \quad 3\to 14,\quad 4\to 15, \quad 5\to 6, \quad 6\to 7,\quad 7\to 3.
$$
Observe
$$
\phi^5(1)=\phi^4(1)\phi^4(2)=\cdots=\phi^4(1)\phi^3(1)\phi^2(1)\phi(1)6
 $$
 Therefore
 $$
 \phi^6(1)=\phi^5(1)\phi^4(1)\phi^3(1)\phi^2(1)7, \quad
\phi^7(1)=\phi^6(1)\phi^5(1)\phi^4(1)\phi^3(1)3.
 $$
Thus,
 $|\phi^n(1)|=2^n$, for $0\leq n\leq 4$,  $|\phi^5(1)|=2^5-1$, and 
 $$|\phi^6(1)|=(2^5-1)+(2^4+2^3+2^2+1)=60$$.
 
 A direct calculation shows that the characteristic polynomial of the matrix $M_{\phi}$ is $p_{7,2}(x)$,
 which implies that 
 the sequence $\{|\phi^n(1)|\}_{n\geq 0}$ satisfies the linear recurrence relation
 $$
 |\phi^{n+7}(1)|=|\phi^{n+6}(1)|+|\phi^{n+5}(1)|+\cdots+|\phi^{n+2}(1)|-|\phi^{n+1}(1)|-|\phi^{n}(1)|,
 $$
 for $n\geq 0$, with the initial values, listed above.
 Therefore the first values of the sequence (as listed in Table~\ref{tab:2}) are
 $$
 1,2,4,8,16,31,60,116,225,436,844,1634,\cdots
 $$
 
 \begin{theorem}\label{thm:2}
 Let $k,r$ be integers with $k\geq 2r+1$, and $r\geq 1$ and $\phi_{k,r}$ be the substitution defined in~(\ref{eqn:phi}).
 \begin{enumerate}
\item[(a)] The substitution $\phi_{k,r}$ has a unique fixed point,  no other periodic orbits.
\item[(b)]  The recurrence of the iterates of the substitutions is given by the relation
$$
\left\{\begin{array}{ccl}
\phi_{k,r}^n(1)&=&\phi_{k,r}^{n-1}(1)\phi_{k,r}^{n-2}(1)\cdots \phi_{k,r}^{n-k+r+1}(1)\phi_{k,r}^{n-k}(r+1)\\
& &\\
\phi_{k,r}^n(r+1)&=&\phi_{k,r}^{n-1}(1)\phi_{k,r}^{n-2}(1)\cdots \phi_{k,r}^{n-k+2r+1}(1)\phi_{k,r}^{n-k+r}(r+1),
\end{array}
\right.
$$
with $n>k$. 
Moreover, it is not possible to express  $\phi_{k,r}^n(1)$ as juxtaposition of words of the form $\phi_{k,r}^j(1)$, with $0\leq j\leq n-1$.
\item[(c)] The characteristic polynomial of the substitution matrix $M_{\phi_{k,r}}$ is $p_{k,r}(x)$.
\item[(d)] The sequence $\{|\phi_{k,r}(n)|\}_{n\geq 0}$ is the $(k,r)$-Lorenz-Fibonacci sequence of the second kind, i.e., $\ell_{k,r}(n)=|\phi_{k,r}^n(1)|$, for all $n\geq 0$.
\end{enumerate}
\end{theorem}
\begin{proof}
In order to simplify the notation and improve the readability, we omit the subindex $(k,r)$ since these values are fixed, and denote $\phi=\phi_{k,r}(n)$.

\smallskip
\noindent
\textbf{(a):}
There exists a fixed point that starts with $1$ 
since $\phi(i)$ starts with $1$ for $1\leq i\leq k-r-1$, and 
$\lim_{n\to\infty}|\phi^n(1)|=\infty$.
There are no other fixed points since 
$\phi(i)$ does not starts with $i$, for $i\geq k-r$.

Let $\vv=v_1v_2\cdots$,  if $2\leq v_1\leq k-r-1$, notice that  $\vv$  is not a periodic point (of period larger than $1$) since $\phi^n(v_1)$ starts with $v_1$ for any $n$.
Consider the situation $k-r\leq v_1< k$, if $ n\geq k-v_1+2$ then  $\phi^n(v_1)$ starts with $1$;
 if $1\leq n< k-v_1+2$, then  $\phi^n(v_1)=v_1+n$,
 and lastly if $1\leq n< k-v_1+2$ then $\phi^{v_1+1}(v_i)=r+1$. Therefore it does not have periodic point of period larger than $1$.
This completes the proof of statement (a).

\medskip
\noindent
\textbf{(b):}
Consider 
$$
\phi^{k}(1)=\phi^{k-1}(1)\phi^{k-2}(2)=
\cdots=\phi^{k-1}(1)\phi^{k-2}(1)\phi^{k-3}(1)\cdots\phi^{r+1}(1)\phi^{r+1}(k-r),
$$
and
\begin{equation}\label{eqn:phi-(k-r)}
\phi^{r+1}(k-r)=\phi^{r+2}(k-r+1)=\cdots=\phi(k)=r+1.
\end{equation}
Therefore:
\begin{equation}\label{eqn:phi_k-1}
\phi^{k}(1)=\phi^{k-1}(1)\phi^{k-2}(1)\phi^{k-3}(1)\cdots\phi^{k-r-1}(1)(r+1).
\end{equation}
On the other hand,
$$
\phi^{k}(r+1)=\phi^{k-1}(1)\phi^{k-1}(r+2)=\cdots=
\phi^{k-1}(1)\phi^{k-2}(1)\cdots\phi^{2r+1}(1)\phi^{2r+1}(k-r),
$$
and by~(\ref{eqn:phi-(k-r)}), we get $\phi^{2r+1}(k-r)=\phi^{r}(r+1)$.
Hence:
\begin{equation}\label{eqn:phi_k-r+1}
\phi^{k}(r+1)=\phi^{k-1}(1)\phi^{k-2}(1)\cdots\phi^{2r}(1)\phi^{r}(r+1)
\end{equation}
After applying $\phi^n$ to identities~(\ref{eqn:phi_k-1}) and~(\ref{eqn:phi_k-r+1}), and relabelling we get the desired expressions.

Observe that $\phi^n(1)$ cannot be expressed entirely as juxtaposition of $\phi^j(1)$, with $0\leq j\leq n-1$, because the symbol $k$ appears in $\phi^n(1)$, only for $n\geq k-1$, and every occurrence of $k$ is followed by $r+1$.
However in any juxtaposition of words $\phi^j(1)$ the letter following the end of each block is always $1$. Thus, this decomposition is not possible.
Proving  statement (b).

\medskip
\noindent
\textbf{(c):}
Let denote $A_m$ the square matrix of  order $m$ such that the entries for the  first row and the lower co-diagonal are 1 and the rest are 0,
$$A_m=\left(
\begin{array}{ccccc}
1 & 1 & \cdots  & 1 & 1 \\
1 & 0 & \cdots  & 0 & 0 \\
\vdots  & \vdots  & \ddots  & \vdots  & \vdots  \\
0 & 0 & \cdots  & 0 & 0 \\
0 & 0 & \cdots  & 1 & 0%
\end{array}%
\right) .$$

Let be $q_m(x)= (-1)^m\det(A_m-xI_m)$  its characteristic polynomial, here $I_m$ denotes the $m\times m$-identity matrix.
Using Laplace's expansion on the last column we have the recursive form  $q_m(x)=q_{m-1}(x)-1$. Then by induction it is easy to prove that    $q_m(x)=x^m - x^{m-1} - \cdots x-1$.
Moreover,  notice that $A_m$  is the companion matrix of the polynomial $q_m(x)$. Notice that $q_m(x)=p_{m,0}(x)$.

Let  $M_{\phi_{k,r}}$ be the associated matrix with the substitution $\phi_{k,r}$, which is a $k\times k$  matrix and it can decomposed in the following blocks:
$M_{\phi_{k,r}}=\left(
\begin{array}{cc}
A_{r+1}&B_{r,k-r}\\
C_{k-r,r}&D_{k-r}
\end{array}%
\right) $
where $A_{r+1}$  and $D_{k-r-1}$ are  a $r+1$-square  and   $(k-r-1)$-square matrices, respectively,  given by
$$ A_{r+1}=\left(\begin{array}{ccccc}
1 & 1 & \cdots  & 1 & 1 \\
1 & 0 & \cdots  & 0 & 0 \\
\vdots  & \ddots  & \ddots  & \vdots  & \vdots  \\
0 & 0 & \ddots  & 0 & 0 \\
0 & 0 & \cdots  & 1 & 0%
\end{array}%
\right) ,
\quad
D_{k-r-1}:= \left(\begin{array}{ccccc}
0 & 0 & \cdots  & 0 & 0 \\
1 & 0 & \cdots  & 0 & 0 \\
\vdots  & \ddots  & \ddots  & \vdots  & \vdots  \\
0 & 0 & \ddots  & 0 & 0 \\
0& 0 & \cdots  & 1 & 0
\end{array}\right);
 $$
and   $B_{r+1,k-r-1}$   a $(r+1)\times (k-r-1)$-matrix
  having the first $k-2r-2$ entries at the first row  equals to $1$ and $1 $ at the $(r+1,k-r-1)$ entry,  and $C_{k-r-1,r+1}$  is a  $(k-r-1) \times (r+1)$-matrix having $1$ at the $(1,r+1)$ entry and $0's$ everywhere. Namely:

$$
\begin{tikzpicture}
\node (B) at (-4,0) {$B_{r+1,k-r-1}:=$};
\node (comma) at (2.9,0) {$,$};

\matrix (m) [matrix of math nodes,left delimiter=(,right delimiter=)] {
1 & 1 & \cdots & 1 & 0 & \cdots & 0 \\
0 & \cdots & 0 & 0 & \cdots & 0 & 0 \\
\vdots & \ddots & \vdots & \vdots & \ddots & \vdots & \vdots \\
0 & \cdots & 0 & 0 & \cdots & 0 & 0 \\
0 & \cdots & 0 & 0 & \cdots & 0 & 1 \\
};

\draw[decorate,decoration={brace,amplitude=5pt}]
(m-1-1.north west) -- (m-1-4.north east)
node[midway,above=6pt] {$k-2r-2$};
\end{tikzpicture}
$$
and
 $$ 
C_{k-r-1,r+1}:= \left(\begin{array}{ccccc}
0 & 0 & \cdots  & 0 & 1 \\
0 & 0 & \cdots  & 0 & 0 \\
\vdots  & \vdots  & \ddots  & \vdots  & \vdots  \\
0 & 0 & \ddots  & 0 & 0
\end{array}\right).$$

Let $Q_{k,r}(x)=\det(M_{k,r}-x I_{k})$, which is given by 

$$Q_{k,r}(x)=(-1)^k
\det\left(\begin{array}{cccccccccccccc}
1-x   & 1    & \cdots&1     & 1     &0     &\cdots& 0    &0    \\
1     & -x   & \cdots&0     & 0     &0     &\cdots& 0    &0   \\
\vdots&\ddots& \ddots&\vdots& \vdots&\vdots&\vdots&\vdots&\vdots\\
0     & 0    & \ddots& -x   &0      &0     &\cdots& 0    &0\\
0     & 0    & \cdots&1     & -x    &0     &\cdots&0     &1 \\
0     & 0    & \cdots&0     & 1     &-x    &\ddots&0     &0\\
\vdots&\vdots& \ddots&\ddots&\vdots &\vdots&\ddots&\vdots&\vdots\\
0     &0     & \cdots&0     &0      &0     &\ddots&-x    &0\\
0     & 0    & \cdots&     0&0      & 0    &\cdots& 1    &-x
\end{array}\right).$$
By Laplace's expansion using the last column, it follows

\begin{multline*}
Q^r_k(x)=(-1)^k(-1)^{k+r+1}\det\left(\begin{array}{cc}
A_k -x I_k & E_{r,k-1-r}\\
O_{k-1-r,k} & I_{k-1-r} -x D^{\tau}_{k-1-r}
\end{array}\right)+\\
+(-1)^{k+1}x\det\left(\begin{array}{cc}
A_{k-r-1} -x I_{k-r-1} & O_{k-r-1,r}\\
C_{r,k-r-1} & D_r-xI_{r}
\end{array}\right),
\end{multline*}
where $D^\tau$ is the transpose of the matrix $D$, $O_{s,t}$ is the $s\times t$-matrix whose entries are equal zero, and $E_{s,t}$ is the $s\times t$-matrix with entries $e_{i,j}$, such that $e_{1,1}=1$ and $e_{i,j}=0$ otherwise.

It yields
\begin{eqnarray*}
Q_{k,r}(x)&=& (-1)^{r+1}\det(A_r-xI_r)\det(I_{k-1-r} -x D^{\tau}_{k-1-r})+\\
& &+(-1)^{k+1}\det(A_{k-r-1} -x I_{k-r-1})\det( D_r-xI_{r})\\
&=&(-1)^{r+1}\det(A_r-xI_r)-x^{r+1}(-1)^{k+r}\det(A_{k-r-1}-xI_{k-r-1}),
\end{eqnarray*}
thus
$Q_{k,r}(x)= x^{r+1} q_{k-r-1}(x)-q_{r}(x)$.
It follows  $$Q_{k,r}(x)= x^{r+1} (x^{k-r-1}-x^{k-r-2}-\cdots -x-1) -(x^{r}-x^{r-1}-\cdots -x-1)$$
and finally $$Q_{k,r}(x)= x^{k}-x^{k-1}-\cdots -x^{r+2}-x^{r+1} -x^{r}+x^{r-1}+\cdots +x+1=p_{k,r}(x).$$
This completes  the proof of statement (c).

\medskip
\noindent
\textbf{(d):}
By (c) the characteristic polynomial of $M_{\phi_{k,r}}$ is $p_{k,r}(x)$, thus
the sequence $\{|\phi^n(1)|\}_{n\geq 0}$ satisfies the linear recurrence relation given in~(\ref{eqn:lr}).  
Hence, to prove the statement it suffices to show that $|\phi^n(1)|=\ell(n)$, for $0\leq n\leq k-1$, which are listed in~(\ref{eqn:ell-ic}).

From the fact $\phi(j)=1(j+1)$, for $1\leq j\leq k-r-2$, it follows $|\phi^n(1)|=2^n$, with  $0\leq n\leq k-r-1$.
Using a similar analysis done in (b) 
$$
\phi^{k-r}(1)=\phi^{k-r-1}(1)\phi^{k-r-2}(1)\cdots\phi(1)(k-r+1).
$$
It yields
$$
|\phi^{k-r}(1)|=\sum_{i=1}^{k-r-1} |\phi^{k-r-i}(1)|+ 1=2^{k-r}-1.
$$
Consider
$$
\phi^{k-r+1}(1)=\phi^{k-r}(1)\phi^{k-r-1}(1)\cdots\phi^2(1)\phi(1)(k-r+2),
$$
hence 
\begin{eqnarray*}
|\phi^{k-r+1}(1)|&=&|\phi^{k-r}(1)|+\left(|\phi^{k-r-1}(1)|+\cdots+|\phi^2(1)|\right)+|(k-r+2)|\\
&=&(2^{k-r}-1)+(2^{k-r-1}+\cdots+4)+1=2^{k-r+1}-4.
\end{eqnarray*}
Proceeding inductively, for $2\leq j\leq r-1$, it yields
\begin{eqnarray*}
|\phi^{k-r+j}(1)|&=&\left(|\phi^{k-r+j-1}(1)|+\cdots+|\phi^{k-r}(1)|\right)+\left(|\phi^{k-r-1}(1)|+\cdots+|\phi^{j+1}(1)|\right)+|(k-r+j+1)|\\
&=&((2^{k-r+j-1}-j2^{j-1})+\cdots+(2^{k-r}-1))+(2^{k-r-1}+\cdots+2^{j+1})+1\\
&=& (2^{k-r+j-1}+\cdots+2^{j+1})-(j2^{j-1}+\cdots+1)+1\\
&=& 2^{k-r+j}-(j+1)2^{j}.
\end{eqnarray*}
Which coincides with the initial values of $\ell_{k,r}(n)$ listed in Table~(\ref{tab:2}).
Therefore $\ell_{k,r}(n)=|\phi_{k,r}^n(1)|$, for all $n\geq 0$.
This completes the proof of statement (d).
\end{proof}


\section{Some remarks}\label{s:remarks}

\begin{enumerate}
\item 
The sequence $\{L_{k,r}(n)\}_n$ is related to the growth of the elements of the Markov partition for the Lorenz maps whose shifts are associated with the polynomial $p_{k,r}(x)$. However this topic is beyond the scope of the present article.

Another important item to explore is how the substitutions $\zeta_{k,r}$ and $\phi_{k,r}$ are related with the dynamics of the Lorenz maps.

\smallskip
\item
We remark that the graph $\G_{\zeta_{k,r}}$ (Figure~\ref{fig:k-r}) is obtained from the graph that define  the  shift in~\cite{SS} by relabelling the  vertices according to the  rule:
$$
1 \to  k, \quad 2\to (k-1),\quad \cdots\quad, (k-1)\to 2, \quad k\to  1,
$$ and inverting the direction of the  edges.
This fact is revealed in the proof of Theorem~\ref{thm:1}(d), since the matrix that encode the relabelling the indices is the permutation matrix $P$ defined in the proof.

\smallskip

\item The graphs $\G_{\phi_{k,r}}$ and $\G_{\zeta_{k,r}}$ are not isomorphic for $r\geq 1$. Since the in-degrees of the vertices are different, which is an invariant of the graph~\cite{distel}.

Note that in $\G_{\phi_{k,r}}$ the in-degree of the vertex $1$, i.e., the number of incoming edges in the vertex $1$ is $k-r$, whereas the in-degree of the other vertices is equal $1$.
On the graph $\G_{\zeta_{k,r}}$ the in-degree of vertex $1$ is $k-2r$ and the in-degree of the vertex $2$ is $2r-2$.  
Therefore $\G_{\phi_{k,r}}$ and $\G_{\zeta_{k,r}}$ are not isomorphic.

\smallskip

\item We consider the substitution $\phi_{k,r}$ over any other substitution sharing the same graph, i.e., $\G_{\phi_{k,r}}$ since the language generated by the fixed point of $\phi_{k,r}$ has lower complexity than the languages generated by the other substitutions. In the present article, we do not deal with this topic, since it is beyond its scope.

\smallskip

\item For the reason pointed out in the previous item, it is  worth  exploring  the languages generated by the different Lorenz-Fibonacci substitutions.
In particular the study of their complexity function.

\smallskip

\item We would like to remark that an alternative approach to the describe the recurrence of the words $\phi_{k,r}^n(1)$ is to introduce the free group generated by the alphabet, and the corresponding homomorphism induced by the substitution in this free group. However, we did not consider this approach in this article.

\end{enumerate}

\section{Conclusions}\label{s:conclusions}

In this article, we introduced and studied two families of integer sequences the $(k,r)$-Lorenz-Fibonacci sequences of the first and second kind, whose characteristic polynomial is $x^k-x^{k-1}-\cdots- x^r+x^{r-1}+\cdots +x+1$.
These families of sequences include the classical $k$-bonacci sequence, as a special case (when $r=0$).
Moreover they arise naturally in the study of the symbolic dynamics of the Lorenz attractor. 
We computed their generating functions (Propositions~\ref{prop:gf} and~\ref{prop:gf2}).

We introduced two families of substitutions (or morphisms)  on an alphabet of $k$ symbols associated with these sequences.
The $(k,r)$-Lorenz-Fibonacci substitutions ($\zeta_{k,r}$), are linked with the first kind sequences, 
whereas  $(k,r)$-Lorenz-Fibonacci direct substitutions ($\phi_{k,r}$) are linked with the sequences of the second kind.
Both substitutions share the same incidence matrix, however their combinatorial properties are very different, their associated graphs are not isomorphic for $r\geq 1$.
These properties are shown in Theorem~\ref{thm:1} and~\ref{thm:2}, respectively.

There are different  open directions for future research, such as describing in detail the languages (including the complexity function) associated with these substitutions.
The connection of the substitutions $\zeta_{k,r}$ and $\phi_{k,r}$ (and their corresponding dynamical system) with the dynamics of the Lorenz attractor remains an open problem.


\end{document}